\documentclass[11pt,twoside]{article}
\usepackage{a4}
\usepackage{amssymb,amsmath,amsthm,latexsym}
\usepackage{amsfonts}
  \usepackage{amsfonts}
\usepackage{graphicx}

\usepackage{ascmac}
\usepackage{comment}
\usepackage{bm}
\usepackage{ytableau}

\usepackage{hyperref}

\usepackage{tikz}

\usepackage{graphicx}
\usepackage{here}
\usepackage{time}
\usepackage[abbrev]{amsrefs}

\usepackage{xcolor}
\usepackage[capitalize,nameinlink,noabbrev,nosort]{cleveref}
\hypersetup{
 colorlinks=true,       
 linkcolor=blue,          
 citecolor=blue,        
 filecolor=blue,      
 urlcolor=blue,           
}

\newtheorem{theorem}{Theorem}[section]

\newtheorem{corollary}[theorem] {Corollary}

\newtheorem{example}[theorem]{Example}
\newtheorem{lemma} [theorem]{Lemma}

\newtheorem{proposition}[theorem]{Proposition}
\newtheorem{remark}[theorem]{Remark}

\newcommand{\N}{\mathbb{Z}_{\ge1}}
\newcommand{\Z}{\mathbb{Z}}
\newcommand{\Q}{\mathbb{Q}}
\newcommand{\R}{\mathbb{R}}
\newcommand{\C}{\mathbb{C}}

\DeclareMathOperator{\ReNew}{Re}
\renewcommand{\Re}{\ReNew}

\vspace{5cm}

\title{Functional equation, upper bounds and analogue of Lindel\"of hypothesis for the Barnes double zeta-function}

\author{Takashi Miyagawa}

\begin{document}
  
\label{'ubf'}  
\setcounter{page}{1}     

\markboth {\hspace*{-9mm} \centerline{\footnotesize \sc
Functional equation, upper bounds and analogue of Lindel\"of hypothesis
}
}
{ \centerline {\footnotesize \sc  
Takashi Miyagawa
} \hspace*{-9mm}              
}

\vspace*{-2cm}

\begin{center}
{ 
{\Large \textbf
{ \sc  Functional equation, upper bounds and analogue of Lindel\"of hypothesis for the Barnes double zeta-function
}
}
\\

\medskip

{\sc Takashi Miyagawa}\\
{\footnotesize Department of Economics and Information Science, Onomichi City University \\ 
1600-2 Hisayamada-cho, Onomichi, Hiroshima, 722-8506, Japan}\\
{\footnotesize e-mail: {\it miyagawa@onomichi-u.ac.jp}}
}
\end{center}

\thispagestyle{empty}

\hrulefill

\begin{abstract}  
{\footnotesize  
The functional equations of the Riemann zeta function, the Hurwitz zeta function, and the Lerch zeta function have been well known for a long time, and there is great importance in studying these zeta functions. For example, fundamental properties such as the upper bounds, the distribution of zeros, and the zero-free regions in the Riemann zeta function derive from functional equations.
		
In this paper, we consider the functional equations for the Barnes double zeta-function  
$ \zeta_2 (s, \alpha ; v, w ) = \sum_{m=0}^\infty \sum_{n=0}^\infty (\alpha+vm+wn)^{-s} $.
Additionally, by applying this functional equation and the Phragm\'en-Lindel\"of convexity principle, we obtain some upper bounds for $ \zeta_2(\sigma + it, \alpha ; v, w) $ with respect to $ t $ as $ t \rightarrow \infty $.
Furthermore, we prove that an analogue of the Lindel"of hypothesis for the Barnes double zeta-function is equivalent to suitable upper bounds for its positive even moments on vertical lines in the critical strip.
}
\end{abstract}
 \hrulefill



	\section{Introduction}
	In this section, we introduce the Barnes double zeta-function, the functional equations 
	of the Hurwitz zeta-function and the Lerch zeta-function, and some upper bounds for the Riemann
	zeta-function.
	
	Let $ s = \sigma + it $ be a complex variable.
	For $ \theta \in \R $ let 
	$ H(\theta) = \{ z = re^{i(\theta + \phi)} \in \C \, | \, r>0, -\pi/2 < \phi < \pi/2 \} $ be the open half plane whose boundary line is perpendicular to $ e^{i \theta}$.
	The Barnes double zeta-function is defined by
	\begin{equation}\label{B_2-zeta}
		\zeta_2 (s, \alpha ; v, w)
		= \sum_{m= 0}^\infty \sum_{n= 0}^\infty
		\frac{1}{(\alpha + v m + wn)^s},
	\end{equation}
	where $ \alpha, v, w \in H(\theta) $ for some $ \theta$. It is known 
	that this series converges absolutely and uniformly on any compact subset
	in $ \{s \in \C \ | \ \mathrm{Re}(s)>2 \}$.
    It is well known that the Barnes double zeta-function
admits a meromorphic continuation to the whole complex plane,
having simple poles at $s=1$ and $s=2$.
	This function was introduced by E. W. Barnes \cite{Barnes1899}
	in the theory of double gamma function, and double series of the 
	form (\ref{B_2-zeta}) was introduced in \cite{Barnes1901}. Furthermore in \cite{Barnes1904}, 
	in connection with the theory of the multiple gamma function, 
    which we will mention in Remark \ref{linear_zeta_H} in the next section (see Remark \ref{linear_zeta_H} in Sec. 2.)
	$r$-multiple series of the form (\ref{B-zeta}) was introduced,
	which we will mention in a Remark \ref{rem:zeta_r}.
Recent studies on the analytic properties of the Barnes double
zeta-function include works of Ruijsenaars
\cite{Ruijsenaars2000}, Spreafico \cite{Spreafico2009},
and the author \cite{Miyagawa2019}.
These works investigate various aspects of the Barnes multiple
zeta-function, including its analytic continuation,
special values, and related multiple gamma functions.
	The above series (\ref{B_2-zeta}) is a double-series version of the Hurwitz zeta-function \cite{LaurincikasGarunkstis2002}
	\begin{equation}\label{zeta_H}
		\zeta_H(s, \alpha) = \sum_{n=0}^\infty \frac{1}{(n+\alpha)^s}
		\qquad (0 < \alpha \leq 1).
	\end{equation}
	Additionally, as a generalization of this series in a different direction, the Lerch zeta-function \cite{LaurincikasGarunkstis2002}
	\begin{equation}\label{zeta_L}
		\zeta_L(s, \alpha, \lambda) = \sum_{n=0}^\infty \frac{e^{2\pi in \lambda}}{(n+\alpha)^s}
		\qquad (0 < \alpha \leq 1, \ 0 < \lambda \leq 1),
	\end{equation}
	is also an important topic of research. These series are absolutely convergent for $ \sigma > 1 $.
	Moreover, if $ 0 < \lambda < 1 $, then the series (\ref{zeta_L}) is convergent even for $ \sigma > 0 $.
	Furthermore, series (\ref{zeta_H}) and (\ref{zeta_L}) satisfy the following functional equations 
	\begin{equation}\label{eq:func-eq_zeta_H}
		\zeta_H(s,\alpha)
		=\frac{\Gamma(1-s)}{(2\pi)^{1-s}}
		\left\{ e^{\pi i(1-s)/2} \sum_{n=1}^\infty \frac{e^{2\pi in(1-\alpha)}}{n^{1-s}}
		+ e^{-\pi i(1-s)/2} \sum_{n=1}^\infty \frac{e^{2\pi in\alpha}}{n^{1-s}} \right\},
	\end{equation}
	as $ \sigma < 0 $, and 
	\begin{eqnarray}
		&& \zeta_L(s, \alpha, \lambda) = \frac{\Gamma(1-s)}{(2\pi)^{1-s}}
		\left\{ e^{\{(1-s)/2 - 2 \alpha \lambda\}\pi i} 
		\sum_{n=0}^\infty \frac{e^{2\pi in(1-\alpha)}}{(n+\lambda)^{1-s}} \right.	\nonumber \\
		&& \qquad \qquad \qquad \qquad \qquad \qquad \quad
		\left.	+e^{\{-(1-s)/2 + 2 \alpha (1-\lambda)\}\pi i } 
		\sum_{n=0}^\infty \frac{e^{2\pi in \alpha}}{(n+1-\lambda)^{1-s}}
		\right\},	\label{eq:func-eq_zeta_L}
	\end{eqnarray}
	with $ \sigma < 0, 0 < \lambda < 1 $, respectively, which are well known (see \cite{LaurincikasGarunkstis2002}).
    Expressing the series on the right in (1.4) and (1.5)
in terms of the Lerch zeta-function, we get
		\begin{equation*}
			\zeta_H(s,\alpha)
			=\frac{\Gamma(1-s)}{(2\pi)^{1-s}}
			\left\{ e^{\pi i(1-s)/2} \zeta_L(1-s, 1, 1-\alpha)
			+ e^{-\pi i(1-s)/2} \zeta_L(1-s, 1, \alpha) \right\},
		\end{equation*}
		\begin{eqnarray*}
			&& \zeta_L(s, \alpha, \lambda) = \frac{\Gamma(1-s)}{(2\pi)^{1-s}}
			\left\{ e^{\{(1-s)/2 - 2 \alpha \lambda\}\pi i} 
			\zeta_L(1-s, \lambda, 1-\alpha) \right.	\nonumber \\
			&& \qquad \qquad \qquad \qquad \qquad \qquad \quad
			\left.	+e^{\{-(1-s)/2 + 2 \alpha (1-\lambda)\}\pi i } 
			\zeta_L(1-s, 1-\lambda, \alpha)
			\right\}.	
		\end{eqnarray*}
		Also, since it is well known that
        $\zeta_H(s,\alpha)$ and $ \zeta_L(s, \alpha, \lambda) $ admit analytic continuations to meromorphic functions in the whole complex plane, the above two equations holds for the whole complex $s$-plane.
	
	Next, we recall some well known order estimates for the Riemann zeta-function
	\[
	\zeta(s) = \sum_{n=1}^\infty \frac{1}{n^s}.
	\]
	These estimates concern the order of $ \zeta(\sigma+it) $ 
	as the imaginary part $t$ tends to infinity.
	In particular, research on the order evaluation for  
	$ \zeta(1/2 + it) $ as $ t \rightarrow \infty $ is particularly prominent.
	
	The classical asymptotic formula
	\[
\zeta(\sigma + it)
\ll
|t|^{(1-\sigma)/2+\varepsilon}
\qquad
(0\le \sigma \le 1,\ |t|\ge2)
\]
is known (see \cite{Ivic1985,Titchmarsh1986}).
The specialization $\sigma=1/2$ gives
\[
\zeta(1/2+it)\ll t^{1/4+\varepsilon}.
\]

	This result is obtained using the functional equation of $ \zeta(s) $ and the Phragm\'en-Lindel\"of convexity principle (Lemma \ref{th:Phragmen-Lindelof}). 
	Another classical asymptotic formula for $ \zeta(s) $
was proved by Hardy and Littlewood
(see \cite[Chapter~1]{Ivic1985}) as follows:
	\begin{equation}\label{order_of_zeta}
		\zeta(s) = \sum_{n \leq x} \frac{1}{n^s} - \frac{x^{1-s}}{1-s} 
		+ O(x^{-\sigma}) \qquad (x \rightarrow \infty),
	\end{equation}
	uniformly for $ \sigma \geq \sigma_0 > 0,\ |t| < 2\pi x/C $, when $ C > 1 $ 
	is a constant. 
	Also, Hardy and Littlewood proved the following asymptotic formula 
	({\S}4 in \cite{Titchmarsh1986}); suppose that 
	$ 0 \leq \sigma \leq 1,\ x \geq 1,\ y\geq 1 $ and $ 2\pi xy = |t| $ then
	\begin{equation}\label{AFE}
		\zeta(s) = \sum_{n \leq x} \frac{1}{n^s} + \chi(s) \sum_{n \leq y} \frac{1}{n^{1-s}}
		+ O(x^{-\sigma}) + O(|t|^{1/2-\sigma}y^{\sigma-1}),
	\end{equation}
	where $ \chi(s) = 2 \Gamma(1-s) \sin{(\pi s/2)}(2\pi)^{s-1} $ and note that the functional equation
	$ \zeta(s) = \chi(s) \zeta(1-s) $ holds. This formula (\ref{AFE}) is called the
	approximate functional equation.
	Hardy and Littlewood improved the order estimete to
	\begin{equation}\label{Hardy-Littlewood}
		\zeta \left( \frac12 + it \right) \ll t^{1/6 + \varepsilon},
	\end{equation}
	by the van der Corput method, applying to (\ref{AFE}).
	In 1988, Bombieri and Iwaniec established the bound $ \zeta(1/2 + it) \ll t^{9/56 + \varepsilon} $ (see \cite{BombieriIwaniec1986}). Since then, many mathematicians have refined this result, with Huxley  proving $ \zeta(1/2 + it) \ll t^{32/205 + \varepsilon} $ in 2005 (see \cite{Huxley2005}).
	Furthermore, in 2017, Bourgain proved the best known bound $ \zeta(1/2+it)\ll t^{13/84+\varepsilon} $ (see \cite{Bourgain2017}).
	
	Improvements in the order estimation of $ \zeta(1/2+it) $ are still ongoing, and the true order is expected to be
	\begin{equation}\label{Lindelof_hypothesis}
		\zeta \left( \frac12 + it \right) \ll t^{\varepsilon},
	\end{equation}
	for any $ \varepsilon > 0 $. 
    The conjecture that \eqref{Lindelof_hypothesis}
holds for any $\varepsilon > 0$ is called the Lindel\"of hypothesis.

	Furthermore, it is well known that the Lindel\"of hypothesis is equivalent to
	\begin{equation}\label{Lindelof_hypothesis_int}
	\int_2^T |\zeta(\sigma + it)|^{2k} dt = O(T^{1+\varepsilon})
	\qquad \left( \frac12 \leq \sigma \leq 1 \right),
	\end{equation}
	for any $ k \in \N $ and any $ \varepsilon > 0 $.


In \cite{Miyagawa2018}, the author proved the analogue of \eqref{order_of_zeta} for the case of the Barnes double zeta-function as follows.
	\medskip
	
	\begin{proposition}[Theorem 3 in \cite{Miyagawa2018}]\label{order_of_Barnes_zeta}
	Let $\alpha>0$, $v>0$, $w>0$, $0<\sigma_1<\sigma_2<2$, $x\ge1$, and $C>1$.
		Suppose $ s = \sigma + it \in \C $ with $ \sigma_1 < \sigma < \sigma_2 $
		and $ |t| \leq 2\pi x/C $. Then
		\begin{align}
		    & \zeta_2(s,\alpha;v,w) = \mathop{\sum\sum}\limits_{0 \le m,n \le x} 
			\frac{1}{(\alpha + vm + wn)^s}		\nonumber \\
			& \qquad\qquad	+ \frac{(\alpha + vx)^{2-s} + (\alpha + wx)^{2-s} - (\alpha + vx + wx)^{2-s}}
			{vw(s-1)(s-2)}  + O(x^{1-\sigma} ),  
		\end{align}
		as $ x \rightarrow \infty $. 
	\end{proposition}

    The purpose of this paper is to establish functional equations for the Barnes double zeta-function $\zeta_2(s,\alpha;v,w)$ under general assumptions on the parameters $v$ and $w$, including both linearly independent and dependent cases over $\mathbb{Q}$. We also derive upper bounds for $\zeta_2(\sigma+it,\alpha;v,w)$ by applying the Phragm\'en--Lindel\"of convexity principle, and discuss an analogue of the Lindel\"of hypothesis in this setting.
	
	\section{Statements of Main Results}
	In this section, we state our main results on the functional equations and upper bounds for the Barnes double zeta-function.     
        In particular, in Theorem \ref{th:Main_Theorem1}, the results differ depending on whether the complex parameters 
        $v,w$ are linearly independent or linearly dependent over $ \Q $.
	Additionally, we obtain upper bounds for $ \zeta_2(\sigma + it,\alpha;v,w) \ (0 \leq \sigma \leq 2) $ with respect to $ t \rightarrow \infty $, and discuss an 
	analogue of the Lindel\"of hypothesis for the Barnes double zeta-function.
	If $ \alpha, v, w \in H(\theta) $, then $ \alpha + vm + wn \in H(\theta) $.
	Here, let $ \alpha = v(1-y_1) + w(1-y_2) $ and consider only the case where $ 0 \leq y_1, y_2 < 1 $.

Throughout this paper, complex powers are defined by
\[
z^{\,1-s}=\exp\!\bigl((1-s)\mathrm{Log}(z)\bigr),
\]
where $\mathrm{Log}(z)$ denotes the principal branch of the logarithm,
defined by $-\pi < \arg z \le \pi$.
In particular, positive real numbers are assigned their usual real powers.

	\begin{theorem}\label{th:Main_Theorem1}
    Suppose that $\alpha, v, w \in H(\theta)$, and write $\alpha = v(1-y_1) + w(1-y_2)$ with $0 \le y_1, y_2 < 1$.
	We have the following;
	\begin{enumerate}
	\item[$(i)$]
	If $v/w$ is either non-real or a real algebraic irrational number, we have 
	\begin{align}
    &\zeta_2(s, \alpha; v,w)		\nonumber \\
	&=- \frac{\Gamma(1-s)}{(2\pi i)^{1-s}e^{\pi is}}\left\{ \frac{1}{v^s} 
	\mathop{\sum_{n=-\infty}}\limits_{n \neq 0}^\infty \frac{e^{2\pi in(vy_1+wy_2)/v}}
	{(e^{2\pi inw/v}-1)n^{1-s}} + \frac{1}{w^s}
	\mathop{\sum_{n=-\infty}}\limits_{n \neq 0}^\infty \frac{e^{2\pi in(vy_1+wy_2)/w}}
	{(e^{2\pi inv/w}-1)n^{1-s}} \right\}. 
	\label{Thm1(i)}
	\end{align}
    The two series on the right-hand side converge absolutely and locally uniformly in $s \in \mathbb{C}$ when $v/w$ is non-real and $0<y_1,y_2<1$.
    In the case where $v/w$ is a real algebraic irrational number, they converge absolutely and locally uniformly in the half-plane $\Re(s)<0$ when $y_1=0$ or $y_2=0$.
	\item[$(ii)$]
	If $ v/w $ is a rational number, i.e. $ v, w $ are linearly dependent over $ \mathbb{Q} $,
			there exist $ p, q \in \mathbb{N} $ such that $ pv=qw $ and $ (p,\;q)=1 $. Then we have
			\begin{align}
				&\zeta_2(s, \alpha; v,w)		\nonumber \\
				&=- \frac{\Gamma(1-s)}{(2\pi i)^{1-s}e^{\pi is}}
				\left\{ \frac{1}{v^s} 
				\mathop{\sum_{n=-\infty}}\limits_{q \; \mid \hspace{-.25em} / \, n}^\infty
				\frac{e^{2\pi in(vy_1+wy_2)/v}}
				{(e^{2\pi inw/v}-1)n^{1-s}} + \frac{1}{w^s}
				\mathop{\sum_{n=-\infty}}\limits_{p \; \mid \hspace{-.25em}/ \, n}^\infty
				\frac{e^{2\pi in(vy_1+wy_2)/w}}
				{(e^{2\pi inv/w}-1)n^{1-s}} \right. \nonumber \\
				& \quad 
				+ \frac{q^{s-1}}{2\pi ipv^s}(s-1) 
				\left( \zeta_L\left(2-s,1,-\frac{q\alpha}v\right)
				+ e^{\pi is}\zeta_L\left(2-s,1,\frac{q\alpha}v \right)
				\right)
				\nonumber \\
				& \quad
				\left.
				+\left( \frac{p+q}{2pv} - \frac{\alpha q}{pv^2} \right)
				\mathop{\mathop{\left( \frac{q}{v} \right)^{s-1}_{}}\limits_{}}
				\limits_{}
				\left( \zeta_L \left(1-s,1,-\frac{q\alpha}v\right) 
				- e^{\pi is}  \zeta_L \left(1-s,1, \frac{q\alpha}v \right) \right)
				\right\}, \nonumber \\ \label{func_eq_dep}
			\end{align}
			and convergence of the infinite series on the first and second terms on the right-hand side are the same as in $(i)$.
		\end{enumerate}
	\end{theorem}

\medskip

\begin{remark}
The assumption on $v/w$ is imposed in order to avoid small divisor
phenomena in the factor
\[
 e^{2\pi i n w/v}-1 .
\]
If $v/w$ is non-real, this denominator can be controlled by the
exponential growth or decay of $e^{2\pi i n w/v}$. 
On the other hand, when $v/w$ is real and irrational, the denominator
may become arbitrarily small. In the present paper we assume that
$v/w$ is real algebraic irrational in this case; then Roth's theorem
gives a sufficient Diophantine lower bound, which ensures the required
convergence. For real transcendental values of $v/w$, such a uniform
Diophantine estimate is not available in general, and therefore this
case is not treated here.
\end{remark}

\medskip

\begin{remark}\label{linear_zeta_H}
	The above functional equation (\ref{func_eq_dep}) is a generalization 
	of the well-known equation 
\begin{equation}
	\zeta_2(s, \alpha; 1,1) = (1-\alpha) \zeta_H(s, \alpha) + \zeta_H(s-1, \alpha)		\label{eq_double_Hurwits}
\end{equation}
for the double Hurwitz zeta-function (See \cite{SrivastavaChoi2001}, p. 86), with the additional parameters $ v$ and $ w $.
In fact, (\ref{func_eq_dep}) can be rewritten as stated in the following corollary, which explicitly shows that the classical identity 
\eqref{eq_double_Hurwits} is recovered as a special case.
\end{remark}

\medskip

\begin{corollary}\label{Cor}
If $v$ and $w$ are linearly dependent over $\mathbb{Q}$, then there exist
$p,q\in\mathbb{N}$ such that $pv=qw$ and $(p,q)=1$. Then we have
\begin{align}
    & \zeta_2(s, \alpha; v,w)		\nonumber \\
	&= \frac{\Gamma(1-s)}{(2\pi)^{1-s}}e^{\pi i(1-s)/2}
	\left\{ \frac{1}{v^s} 
	\mathop{\sum_{n=-\infty}}\limits_{q \; \mid \hspace{-.25em} / \, n}^\infty
	\frac{e^{2\pi in(vy_1+wy_2)/v}}
	{(e^{2\pi inw/v}-1)n^{1-s}} + \frac{1}{w^s}
	\mathop{\sum_{n=-\infty}}\limits_{p \; \mid \hspace{-.25em}/ \, n}^\infty
	\frac{e^{2\pi in(vy_1+wy_2)/w}}
	{(e^{2\pi inv/w}-1)n^{1-s}} \right\} \nonumber \\
	& \quad 
	+ \left( \frac{p+q}{2pv} - \frac{vp}{2q} \right)
	\left( \frac{q}{v} \right)^{s-1}
	\zeta_H \left(s, \frac{q\alpha}v \right) 
	+ \frac{q^{s-1}}{pv^s}
	\zeta_H\left(s-1, \frac{q\alpha}v \right).
	\label{cor3} 
\end{align}
The two series on the right-hand side converge absolutely and locally
uniformly in the half-plane $\mathrm{Re}(s)<0$ for $0 \le y_1,y_2<1$.
\end{corollary}

\medskip

\begin{remark}\label{rem:zeta_r}
	The functional equation for Theorem \ref{th:Main_Theorem1}(i) was given by Y. Komori, K. Matsumoto, H. Tsumura \cite{KomoriMatsumotoTsumura2013} for the general $r$-multiple zeta-function  $ \zeta_r(s, \alpha ; w_1, \ldots , w_r) $.
		Its definition is given as follows. Let $ r $ be a positive integer, and let $ w_j \in H(\theta) \ (j = 1, \ldots , r) $ be complex parameters.
		The Barnes multiple zeta-function is defined by
		\begin{equation}\label{B-zeta}
			\zeta_r (s, \alpha ; w_1, \ldots , w_r)
			= \sum_{m_1 = 0}^\infty \cdots \sum_{m_r = 0}^\infty
			\frac{1}{(\alpha + w_1 m_1 + \cdots + w_r m_r)^s},
		\end{equation}
		where the series on the right-hand side is absolutely convergent for $ \mathrm{Re}(s) > r $, and can be continued meromorphically to $ \C $ and its only singularities are the simple poles located at $ s = j \ (j = 1, \ldots , r) $.
		In their result in \cite{KomoriMatsumotoTsumura2013}, the complex parameters $ w_1, w_2, \cdots , w_r $ are given under the condition that they satisfy $ \mathrm{Im}(w_j/w_k) \neq 0 \ (j \neq k) $. 
		However, in the results of Theorem \ref{th:Main_Theorem1}, $v$ and $w$ are classified into linearly independent and linearly dependent cases over $\Q$, extending their results to include the rational dependence case.
		
	\end{remark}
	
\medskip

\begin{theorem}[Convexity bound]\label{th:convexity}
Suppose that $\alpha>0,v>0$ and $w>0$.
\begin{enumerate}
\item[(i)]
If $v/w$ is a real algebraic irrational number, then
\[
\zeta_2(\sigma+it,\alpha;v,w)
\ll
|t|^{(2-\sigma)/4+\varepsilon}\qquad
(0\le \sigma \le 2).
\]
\item[(ii)]
If $v/w\in\mathbb Q$, then
\[
\zeta_2(\sigma+it,\alpha;v,w)
\ll
|t|^{3(2-\sigma)/4+\varepsilon}\qquad
(0\le \sigma \le 2).
\]
\end{enumerate}
\end{theorem}

\medskip
Taking $ \sigma=1, 3/2 $ in this theorem, we obtain the following corollary.	
\medskip 

\begin{example}\label{Cor_2nd}
If $v/w$ is a real algebraic irrational number, then
\[
\zeta_2(1+it,\alpha;v,w)
\ll
|t|^{1/4+\varepsilon},
\]
and
\[
\zeta_2\!\left(\frac32+it,\alpha;v,w\right)
\ll
|t|^{1/8+\varepsilon},
\]
as $t \to \infty$.
\end{example}

\medskip

As an analogue of the equivalence between
\eqref{Lindelof_hypothesis} and
\eqref{Lindelof_hypothesis_int}
for the classical Lindel\"of hypothesis, we obtain the following theorem for the Barnes double zeta-function.

\medskip
   
\begin{theorem}\label{Lindelof_zeta_2}
Suppose that $v>0$, $w>0$, and
$v/w\in\overline{\Q}\setminus\Q$.
For any $\varepsilon>0$,
\begin{equation}\label{eq:LH}
\zeta_2(\sigma+it,\alpha;v,w)
=
O(t^\varepsilon)
\qquad
\left(\frac12\le\sigma\le2\right)
\end{equation}
holds if and only if
\begin{equation}\label{eq:MV}
\int_2^T
|\zeta_2(\sigma+it,\alpha;v,w)|^{2k}\,dt
=
O(T^{1+\varepsilon})
\qquad
\left(\frac12\le\sigma\le2\right)
\end{equation}
holds for any $k\in\N$ and any $\varepsilon>0$.
Moreover, if the implied constant in either \eqref{eq:LH} or \eqref{eq:MV} depends only on $\varepsilon$ and $k$, then the same is true for the other estimate.
\end{theorem}

\medskip
	
\section{Auxiliary lemmas}
	In this section, we collect several auxiliary lemmas that will be used in the proofs of the main theorems.
	Lemma \ref{A1,Lem1}, \ref{lem_2nd} and \ref{lem_3rd} provide supplementary results on the absolute convergence of infinite series involving exponential terms.
	Lemma \ref{th:Phragmen-Lindelof} states the Phragm\'en-Lindel\"of convexity principle,
        which allows us to derive a simple upper bound.
	Finally, Lemma \ref{Lem:int} establishes the equivalence of the Lindel\"of conjecture for the Barnes double zeta-function.
	
	\medskip
	\begin{lemma}[Lemma 1 in \cite{Arakawa1982}]\label{A1,Lem1}
		For any irrational real algebraic number $ \beta_0 $ and any pair of real numbers $ (p_1,p_2) $, the infinite series 
		\begin{equation}
			\eta(\beta_0, s, p_1, p_2) = \sum_{n=1}^\infty \frac{e^{2\pi i n(p_1\beta_0 +p_2)}}{(1-e^{2\pi i n\beta_0})n^{1-s}}
		\end{equation}
		converges absolutely for $\Re(s)<0$.
	\end{lemma}

\medskip

\begin{lemma}\label{lem_2nd}
Let $\alpha,v,w\in H(\theta)$, and suppose that
$\alpha=v(1-y_1)+w(1-y_2) \ \ (0\le y_1,y_2<1)$.
Assume that $w/v$ is either non-real or a real algebraic irrational number.
Then the series
\[
 \sum_{\substack{n=-\infty\\ n\ne0}}^\infty
 \frac{e^{2\pi i n(vy_1+wy_2)/v}}
 {(e^{2\pi i n w/v}-1)n^{1-s}}
\]
has the following convergence properties:
\begin{enumerate}
\item[$(i)$] If $w/v$ is a real algebraic irrational number, then the series
converges absolutely for $\sigma<0$.
\item[$(ii)$] If $w/v$ is non-real, then the series converges absolutely for each fixed $s \in \mathbb{C}$, and the convergence is locally uniform in $s$ when $0<y_2<1$, and for $\sigma<0$ when $y_2=0$.
\end{enumerate}
\end{lemma}
	
\begin{proof}
We first rewrite the series as
\begin{align}
\sum_{\substack{n=-\infty\\ n\ne0}}^\infty 
\frac{e^{2\pi i n (vy_1+wy_2)/v}}{(e^{2\pi i n w/v}-1)n^{1-s}}
&=
\sum_{n=1}^\infty
\left\{
\frac{e^{2\pi i n (vy_1+wy_2)/v}}{(e^{2\pi i n w/v}-1)n^{1-s}} 
+
\frac{e^{-2\pi i n (vy_1+wy_2)/v}}{(e^{-2\pi i n w/v}-1)n^{1-s}}
\right\}
\nonumber \\
&= -\eta(w/v, s, y_2, y_1) 
+ \eta(-w/v, s, y_2, -y_1),
\label{eta}
\end{align}
where the functions $\eta(\cdot)$ are defined in Lemma~\ref{A1,Lem1}.

\noindent
(i) \ 
Suppose that $w/v$ is a real algebraic irrational number.  
By Lemma~\ref{A1,Lem1}, both series on the right-hand side of \eqref{eta}
are absolutely convergent for $\sigma<0$. 

\noindent
(ii) \ 
Suppose that $w/v$ is non-real.  
We consider the case $\mathrm{Im}(w/v)>0$; the case $\mathrm{Im}(w/v)<0$ is similar.
Put $\mathrm{Im}(w/v)=I>0$.
For $z\in\mathbb{C}$ and $\varepsilon>0$, let $D(z,\varepsilon)$ denote the
closed disk centered at $z$ with radius $\varepsilon$.
Then there exists a constant $M>0$ such that
\[
\left|\frac{1}{e^z-1}\right|
\le M e^{-(\mathrm{Re}(z))_+}
\]
for all $z\in\mathbb{C} \setminus \bigcup_{m\in\mathbb{Z}} D(2\pi i m,\varepsilon)$.
Using this estimate, we obtain
\[
\left|
\frac{e^{2\pi i n (vy_1+wy_2)/v}}{e^{2\pi i n w/v}-1}
\right|
\le
\begin{cases}
M_0 e^{-2\pi n I y_2} & (n>0),\\
M_0 e^{2\pi n I (1-y_2)} & (n<0),
\end{cases}
\]
for some constant $M_0>0$.
\\
If $0<y_2<1$, then
\begin{align*}
\sum_{\substack{n=-\infty\\ n\ne0}}^\infty
\left|
\frac{e^{2\pi i n(vy_1+wy_2)/v}}
{(e^{2\pi i n w/v}-1)n^{1-s}}
\right|
\le
M_0\sum_{n=1}^\infty
\left(
\frac{e^{-2\pi n I y_2}}{n^{1-\sigma}}
+
\frac{e^{-2\pi n I(1-y_2)}}{n^{1-\sigma}}
\right).
\end{align*}
Since both exponential factors decay exponentially, the series on the
right-hand side converges for every fixed $s\in\mathbb{C}$.
Moreover, the convergence is locally uniform in $s$.
\\
If $y_2=0$, then
\[
\sum_{\substack{n=-\infty\\ n\ne0}}^\infty
\left|
\frac{e^{2\pi i n(vy_1+wy_2)/v}}
{(e^{2\pi i n w/v}-1)n^{1-s}}
\right|
\le
M_0\sum_{n=1}^\infty
\left(
\frac{1}{n^{1-\sigma}}
+
\frac{e^{-2\pi n I}}{n^{1-\sigma}}
\right),
\]
which converges for $\sigma<0$.
\end{proof}
	\medskip

\begin{lemma}\label{lem_3rd}
	Let $ \alpha, v, w \in H(\theta) $, and suppose they are expressed as $ \alpha = v(1-y_1)+w(1-y_2) \ (0 \leq y_1, y_2 < 1) $. Also, assume that $ v,w $ can be expressed as $ pv = qw $ with $ p,q \in \N $ and $(p,q)=1$. The infinite series
	\begin{equation}\label{series:roth2}
		\mathop{\sum_{n=-\infty}}\limits_{q \; \mid \hspace{-.25em} / \, n}^\infty 
		\frac{e^{2\pi in (vy_1+wy_2)/v}}{(e^{2\pi in p/q}-1)n^{1-s}}
	\end{equation}
	is absolutely convergent for all $ s \in \C $ if $ 0 < y_2 < 1 $, and for $ \sigma<0 $ if $ y_2 = 0 $.
\end{lemma}
The proof is analogous to that of Lemma \ref{lem_2nd}, and is therefore omitted.
	
\medskip

\begin{lemma}[Phragmén--Lindelöf principle]\label{th:Phragmen-Lindelof}
Let
$B=\{s\in\mathbb C \mid \sigma_1\le \Re(s)\le \sigma_2,\ |t|\ge t_0\}.$
Suppose that $f(s)$ is analytic in the interior of $B$ and continuous on $B$.
Assume that
\[
|f(\sigma_1+it)|\ll |t|^a,
\qquad
|f(\sigma_2+it)|\ll |t|^b,
\]
for $|t|\ge t_0$.
Then
\[
|f(\sigma+it)|
\ll
|t|^{\frac{\sigma_2-\sigma}{\sigma_2-\sigma_1}a
+
\frac{\sigma-\sigma_1}{\sigma_2-\sigma_1}b}
\]
for $\sigma_1\le \sigma\le \sigma_2$.
\end{lemma}

The above lemma is a standard consequence of the
Phragmén--Lindelöf convexity principle (see \cite[Section~5.6]{Titchmarsh1986}).
    
\medskip
\begin{lemma}\label{Lem:int}
	Suppose that $ \alpha, v, w \in H(\theta) $, and $ v, w $ are linearly independent over $ \Q $.
	Let $ 1/2 \le \sigma \le 2, t \ge 2 $ and $ \ell \in \N $.
	Then
	\begin{equation}
		|\zeta_2(\sigma +it, \alpha; v,w )|^\ell
		\ll (\log{t}) \int_{-\delta}^\delta
		|\zeta_2(\sigma +it+ix, \alpha; v,w )|^\ell dx + t^{-A},
		\label{power_of_l}
	\end{equation}
	for any fixed $ 0 < \delta < 1/2 $ and $ A > 0 $.
\end{lemma}
	\begin{proof}
		Let $ B, C>0 $, and $ D=[C \log{t}] $ with $ s = \sigma +it $.
		By residue theorem, 
		\begin{align*}
			&\int_0^B \cdots \int_0^B \int_{|z|=\delta}
			\zeta_2(s+z,\alpha;v,w)^\ell e^{(u_1+\cdots+u_D)z}z^{-1}
			dz du_1 \cdots du_D		\\
			&\quad = 2\pi i \ \zeta_2(s, \alpha;v,w)^\ell
			\int_0^B \cdots \int_0^B du_1 \cdots du_D	\\
			&\quad = 2\pi i \ B^D \zeta_2(s, \alpha; v,w)^\ell
		\end{align*}
		holds.
		On the other hand, the circular path $ |z| = \delta $ is divided two parts $ \mathrm{Re}(z)<0 $ and $ \mathrm{Re}(z) \geq 0 $, which we denote by $ R_1 $ and $R_2$, respectively.
		The integral for $ R_1 $ is
		\begin{align*}
			&	\int_{R_1} \zeta_2(s+z, \alpha; v,w)^\ell
			\int_0^B e^{zu_1}du_1
			\int_0^B e^{zu_2}du_2
			\cdots
			\int_0^B e^{zu_D}du_D\ \cdot \frac{dz}z	\\
			& = \int_{R_1} \zeta_2(s+z, \alpha; v,w)^\ell
			\left( \frac{e^{Bz}-1}z \right)^D \frac{dz}z,
		\end{align*}
		where $ |e^{Bz}| \leq 1 $ on $ R_1 $, satisfies $ |e^{Bz}-1| \leq 2 $.
		We denote the right-hand side by $ I_0 $.
		Also, since $ 1/2 \leq \sigma \leq 2 $ and $ -1/2 \leq \mathrm{Re}(z) < 0 $ on $ R_1 $, so $ 0 \leq \mathrm{Re}(s+z) < 2 $. 
		Then we have
		\[
		\zeta_2(s+z, \alpha; v,w) \ll 
		t^{\{2-\mathrm{Re}(s+z)\}/4+\varepsilon} \ll t^{1/2}
		\]
		Then, the above equation can be evaluated as
		\[
		\ll \int_{R_1} t^{\ell/2} \cdot \left(\frac2\delta\right)^D |dz|
		\ll t^{\ell/2} \left(\frac2\delta\right)^D.
		\] 
		Also, let $ U = \exp(u_1+u_2+\cdots+u_D) $. We consider separately as  
		\[
		\frac{U^z}z = \frac{U^z - U^{-z}}z + \frac{U^{-z}}z.
		\]
		The integral for $ R_2 $ is
		\begin{align*}
			&\int_0^B \cdots \int_0^B \int_{R_2}
			\zeta_2(s+z,\alpha;v,w)^\ell e^{(u_1+\cdots+u_D)z}z^{-1}
			dz du_1 \cdots du_D		\\
			&\quad = \int_{R_2} \int_0^B \cdots \int_0^B 
			\zeta_2(s+z,\alpha;v,w)^\ell \left(\frac{U^z - U^{-z}}z + \frac{U^{-z}}z \right)
			du_1 \cdots du_D dz		\\
			&\quad = \int_{R_2} \int_0^B \cdots \int_0^B 
			\zeta_2(s+z,\alpha;v,w)^\ell \left(\frac{U^z - U^{-z}}z \right)
			du_1 \cdots du_D dz		\\
			&\qquad \quad + \int_{R_2} \int_0^B \cdots \int_0^B 
			\zeta_2(s+z,\alpha;v,w)^\ell \left(\frac{U^{-z}}z \right)
			du_1 \cdots du_D dz.
		\end{align*}
		We denote the right-hand side by $ I_1 + I_2 $.
		Since $ |e^{-Bz} -1 | \leq 2 $ on $ z \in R_2 $.
		Also, since $ 1/2 \leq \sigma \leq 2 $ and $ 0 \leq \mathrm{Re}(z) \leq 1/2 $ on $ R_2 $, so $ 1/2 \leq \mathrm{Re}(s+z) \leq 3/2 $. 
		Then we have
		\[
		\zeta_2(s+z, \alpha; v,w) \ll 
		t^{\{2-\mathrm{Re}(s+z)\}/4+\varepsilon} \ll t^{3/8}.
		\]
		Then, $ I_1 $ can be evaluated as
		\begin{align*}
			I_1 
			&= \int_{R_2} \zeta_2(s+z, \alpha; v,w)^\ell
			\left( \frac{1-e^{-Bz}}z \right)^D \frac1z dz
			\ll \int_{R_2} t^{3\ell/8} \cdot \left(\frac2\delta\right)^D |dz| \\
			&\ll t^{3\ell/8} \left(\frac2\delta\right)^D,
		\end{align*}
by Theorem \ref{th:convexity}.
	Next, we consider $ I_2 $. Since $ (U^z - U^{-z})/z $ has a removable singularity at $ z = 0 $, so it is holomorphic at $z=0$. By using Cauchy's integral theorem, the contour $ R_2 $ can be deformed into a path from $ -i\delta $ to $ i \delta $ along the imaginary axis. 
	Then we have 
\begin{align*}
	&\int_0^B \cdots \int_0^B \int_{-i\delta}^{i\delta}
	\zeta_2(s+z,\alpha;v,w)^\ell \left(\frac{U^z - U^{-z}}z \right)
	dz du_1 \cdots du_D		\\
	& \ll \int_0^B \cdots \int_0^B \left| \int_{-i\delta}^{i\delta}
	\zeta_2(s+z,\alpha;v,w)^\ell \left(\frac{U^z - U^{-z}}z \right)
	dz \right|
	du_1 \cdots du_D
\end{align*}
		Here, if $ z = ix (-\delta \leq x \leq \delta) $, then
		\begin{align*}
			\left| \frac{U^z - U^{-z}}z \right|
			&= \left| \frac{U^{ix} - U^{-ix}}x \right|
			= \frac2{|x|} \left| \frac{e^{ix\log{U}} - e^{-ix \log{U}}}{2i} \right| 
			= 2 \left| \frac{\sin{(x \log{U})}}x \right|	\\
			& \ll \log{U} = u_1 + u_2 + \cdots + u_D \leq BD	\\
			& \ll BC\log{t}.
		\end{align*}
		Then we have
		\[
		I_2 \ll B^{D+1} C \int_{-\delta}^\delta 
		\left|\zeta_2(\sigma+it+ix,\alpha;v,w) \right|^\ell(\log{t}) dx.
		\]
		Using the evaluation results of $ I_0, I_1 $ and $ I_2 $, we obtain
		\begin{align*}
			|\zeta_2(\sigma+it,\alpha;v,w)|^\ell
			&\leq \frac1{2\pi i} \cdot \frac1{B^D}(I_2 + I_1 + I_0)	\\
			&\ll BC \int_{-\delta}^\delta 
			\left|\zeta_2(\sigma+it+ix,\alpha;v,w) \right|^\ell(\log{t}) dx
			+ t^{\ell/2} \left(\frac2{B\delta}\right)^D.
		\end{align*}
		By setting $ B = 4/ \delta $, the second term on the rightmost side of the above equation can be evaluated as
		\[
		t^{\ell/2} \left(\frac2{B\delta}\right)^D = t^{\ell} \cdot 2^{-D}
		\ll t^{\ell/2 - C \log{2}}.
		\]
		Furthermore, if we choose $C$ such that $ C> \ell/ \log{4} $, we obtain
		\[
		|\zeta_2(\sigma+it,\alpha;v,w)|^\ell
		\ll (\log{t}) \int_{-\delta}^\delta 
		\left|\zeta_2(\sigma+it+ix,\alpha;v,w) \right|^\ell dx
		+ t^{-A}.
		\]
		as $ A = C \log{2} - \ell/2 $.
	\end{proof}
	
	\section{Proof of Main Results}
	
	In this section, we provide proofs of Theorem \ref{th:Main_Theorem1}, Theorem \ref{th:convexity} and Theorem \ref{Lindelof_zeta_2}.
	First, we present a proof of Theorem \ref{th:Main_Theorem1}. In the course of this proof, we use Lemmas \ref{lem_2nd} and \ref{lem_3rd} to  discuss the absolute convergence of certain infinite series.
	Finally, we give proofs of Theorem \ref{th:convexity} and Theorem \ref{Lindelof_zeta_2}, which rely on Lemma \ref{th:Phragmen-Lindelof} and Lemma \ref{Lem:int}, respectively.
	
	\medskip
	 
	\begin{proof}[proof of Theorem \ref{th:Main_Theorem1}]
	Assume that $ \mathrm{Re}(s)>2 $. For $ x \in H(\theta) $, we have the following formula for the gamma function:
	\[
	x^{-s} = \frac1{\Gamma(s)} \int_0^{e^{-i\theta}\infty} e^{-xu}u^{s-1}du.
	\]
	Then we obtain
	\begin{align}
	    \zeta_2(s, \alpha; v, w)
		&= \sum_{m=0}^\infty \sum_{n=0}^\infty 
		\frac{1}{(\alpha + v m + w n)^s} 	\nonumber \\
		&= \sum_{m=0}^\infty \sum_{n=0}^\infty
		\frac1{\Gamma(s)} \int_0^{e^{-i \theta}\infty} z^{s-1} e^{-(\alpha+vm+wn)z}dz		\nonumber \\
		&= \frac1{\Gamma(s)} \int_0^{e^{-i \theta}\infty} 	
		\frac{z^{s-1} e^{(v+w-\alpha)z}}{(e^{vz}-1)(e^{wz}-1)} dz	\nonumber \\
		&= \frac1{\Gamma(s)(e^{2\pi is}-1)}
		\int_{C_r(\theta)}	
		\frac{z^{s-1} e^{(v+w-\alpha)z}}{(e^{vz}-1)(e^{wz}-1)} dz
		\label{contour_of_Barnes},
	\end{align}
	where the argument of $z$ is taken in the range $ -\theta \leq \arg{z} \leq -\theta + 2\pi $ and $ C_r(\theta) $ is a contour which starts at $ e^{-i \theta}\infty $, moves counterclockwise around the origin with a sufficiently small radius, and ends at $ e^{-i \theta}\infty $.
    \begin{center}
    \includegraphics[
  width=0.6\textwidth,
  trim=3cm 5cm 2cm 3.5cm, 
  clip
    ]{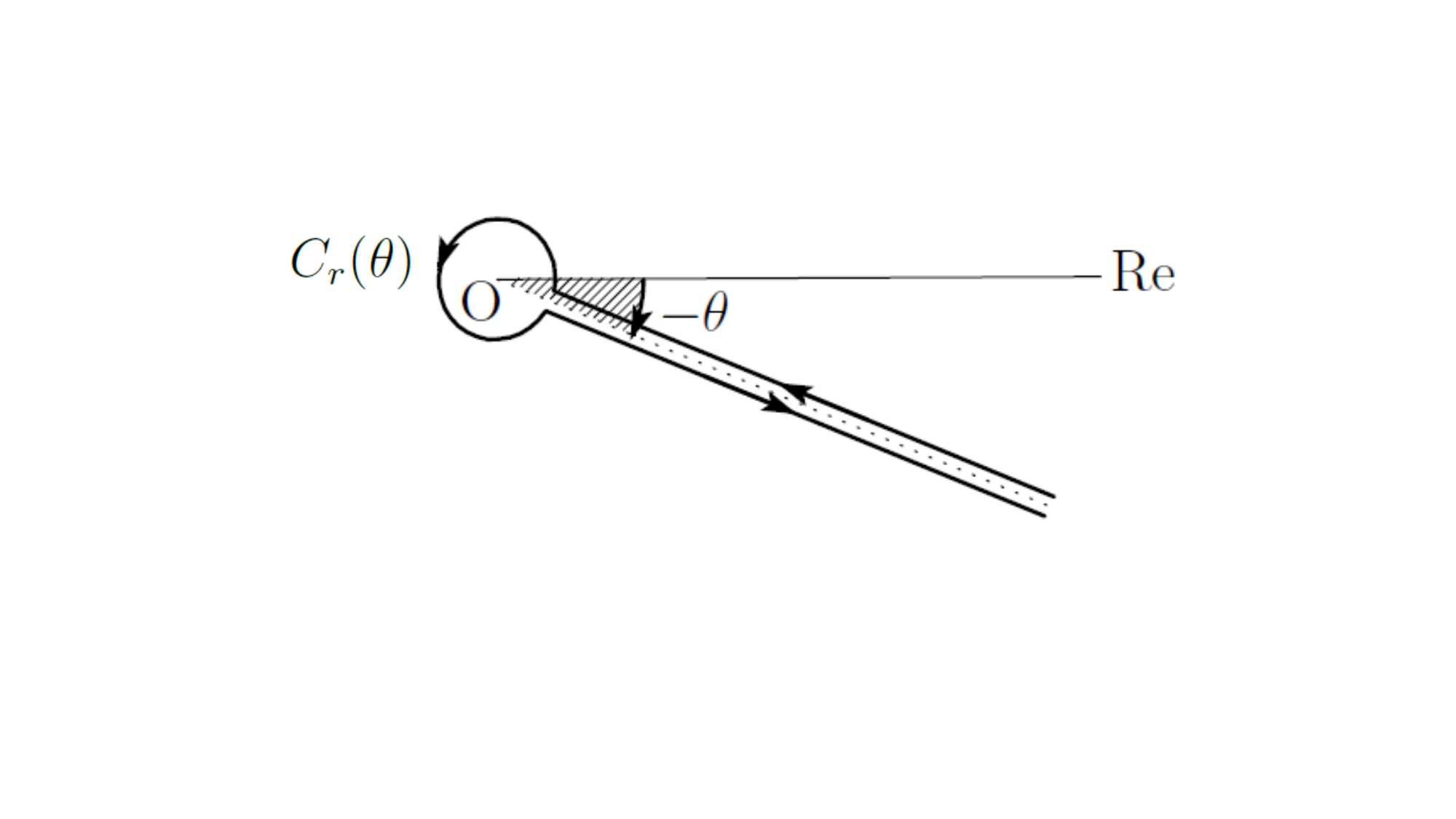}
    \end{center}
    %
	Next, we consider deforming the contour $ C_r(\theta)$ to a new integration path $ C_R(\theta) $ that includes a large circle of radius $ R $. We also examine the residues of the integrand in (\ref{contour_of_Barnes}) given by
	\[
	F(z) = \frac{z^{s-1}e^{(v+w-\alpha)z}}{(e^{vz}-1)(e^{wz}-1)},
	\]
	where, $R$ is a sufficiently large positive real number such that $ C_R(\theta) $ does not pass through any singular point of $F(z)$.
	Note that the residues of $ F(z) $ depend on whether the parameters $ v$ and $ w $ are linearly independent or linearly dependent over $\Q$.

    \begin{center}
    \includegraphics[
  width=\textwidth,
  trim=0cm 1cm 0cm 1cm, 
  clip
    ]{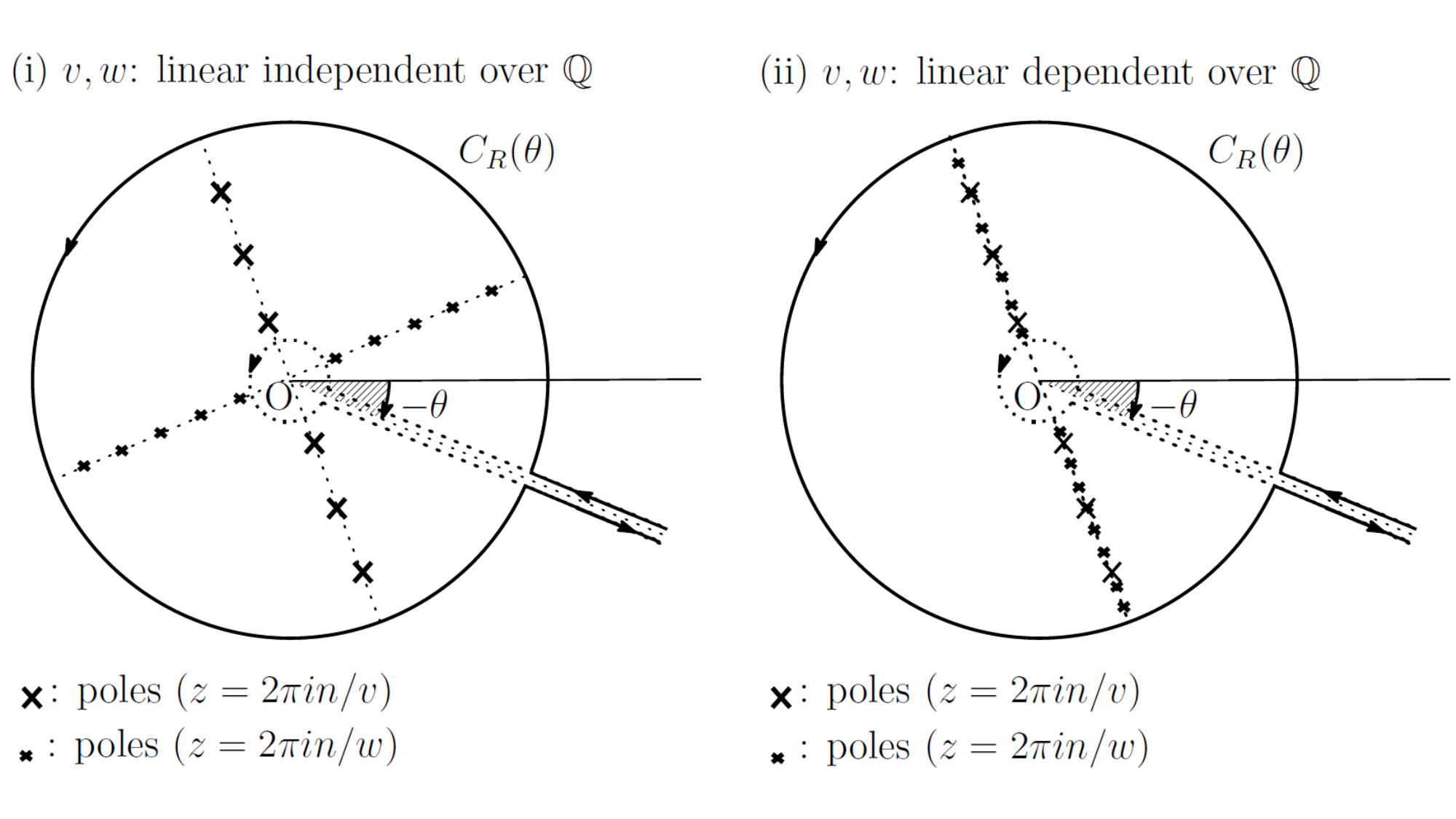}
    \end{center}
	%
	\bigskip
	(i) In the case where $ v, w $ are linearly independent over $ \Q $, and $ v/w $ is an imaginary or irrational real algebraic number,
		$ F(z) $ has simple poles at 	
		\[
		z = \frac{2\pi in}{v},\ \frac{2\pi in}{w} \ (n = \pm1,\ \pm2,\ \cdots).
		\]
		Since
		\begin{align*}
		    \lim_{z \rightarrow 2\pi in/v}\left(z - \frac{2\pi in}{v}\right) F(z)
			&= \lim_{z \rightarrow 2\pi in/v}\left(z - \frac{2\pi in}{v}\right)	
			\frac{z^{s-1}e^{(v + w -\alpha)z}}{(e^{vz}-1)(e^{wz}-1)}	\\
			&= \lim_{z \rightarrow 2\pi in/v} \left( \frac{e^{vz}-e^{2\pi in}}{z-2\pi in/v} \right)^{-1}
			\frac{z^{s-1} e^{(v + w -\alpha)z}}{e^{wz}-1}	\\
			&= \frac{1}{v} \left(\frac{2\pi in}{v} \right)^{s-1}
			\frac{e^{(w-\alpha)2\pi in/v}}{e^{2\pi inw/v}-1},
		\end{align*}
		thus, we have
		\begin{align*}
			\mathop{\mathrm{Res}}\limits_{z = 2\pi in/v} F(z) 
			= \frac{1}{v} \left(\frac{2\pi in}{v} \right)^{s-1}
			\frac{e^{(w-\alpha)2\pi in/v}}{e^{2\pi inw/v}-1}, \quad
			\mathop{\mathrm{Res}}\limits_{z = 2\pi in/w} F(z) 
			= \frac{1}{w} \left(\frac{2\pi in}{w} \right)^{s-1}
			\frac{e^{(v-\alpha)2\pi in/w}}{e^{2\pi inv/w}-1}
		\end{align*}
		and we obtain
		\begin{align}
			& \zeta_2 (s, \alpha ; v,w)	\nonumber \\
			& = - \frac{\Gamma(1-s)}{(2\pi i)^{1-s}e^{\pi is}}
			\left\{ \frac{1}{v^s} \sum_{0<|n|\leq K} 
			\frac{e^{2\pi in(w-\alpha)/v}}{(e^{2\pi inw/v}-1)n^{1-s}} 
			+ \frac{1}{w^s} \sum_{0<|n|\leq L} \frac{e^{2\pi in(v-\alpha)/w}}{(e^{2\pi inv/w}-1)n^{1-s}} \right\}
			\nonumber \\
			& \qquad \qquad  + \frac{1}{\Gamma(s)(e^{2\pi is}-1)} \int_{C_R(\theta)}
			\frac{z^{s-1}e^{(v+w-\alpha)z}}{(e^{vz} -1)(e^{wz}-1)} dz, 	 \label{integral_CR}
		\end{align}
		where $ K = \lfloor|v|R/(2\pi)\rfloor $ and $ L = \lfloor|w|R/(2\pi)\rfloor  $.
		Hence, for
		\[
		z \in \C \backslash \bigcup_{m \in \Z} 
		D(2\pi im/v, \varepsilon) \cup D(2\pi im/w, \varepsilon)
		\]
		we have
		\begin{align*}
			\qquad 
			\left| \frac{z^{s-1}e^{(v+w-\alpha)z}}{(e^{vz} -1)(e^{wz}-1)} \right|
			&= \left| \frac1{e^{vz}-1} \right| \left| \frac1{e^{wz}-1} \right|	|e^{(v+w-\alpha)z}||z^{s-1}|	\\
			&\leq M_1 e^{-(\mathrm{Re}(vz))_+} \cdot M_2 e^{-(\mathrm{Re}(wz))_+}
			e^{\mathrm{Re}((v+w-\alpha)z)}|z|^{\sigma-1}	\\
			&= M_1M_2 \exp{\{\mathrm{Re}((v+w-\alpha)z)-\mathrm{Re}(vz)_+-\mathrm{Re}(wz)_+\}}
			|z|^{\sigma-1}
		\end{align*}
		with a certain $ M_1=M_1(\varepsilon, v)>0 $ and $ M_2=M_2(\varepsilon, w)>0 $.
		Here, let $ 0 \leqq y_1, y_2 < 1 $ and, put
		\[
		\alpha = \alpha(y_1, y_2) = v(1-y_1)+w(1-y_2) \in H(\theta)
		\]
		and $ z' = z/|z| $. Then we see that there exists $ T=T(\varepsilon) \geq 0 $ such that, for
		\begin{align*}
			& \mathrm{Re}((v+w-\alpha)z')- \mathrm{Re}(vz')_+ - \mathrm{Re}(wz')_+	\\
			&\qquad = \mathrm{Re}((vy_1+wy_2)z')- \mathrm{Re}(vz')_+ - \mathrm{Re}(wz')_+	\\
			&\qquad \leq
			\mathrm{Re}(vz') y_1 + \mathrm{Re}(wz') y_2
			- \mathrm{Re}(vz')_+ - \mathrm{Re}(wz')_+	\\
			&\qquad \leq 
			-T,
		\end{align*}
		and so
		\[
		\mathrm{Re}((v+w-\alpha)z) - \mathrm{Re}(vz)_+ - \mathrm{Re}(wz)_+ \leq -T|z|.
		\]
		Hence we see that for all $ z \in \C - \bigcup_{m \in \Z} 
		D(2\pi im/v, \varepsilon) \cup D(2\pi im/w, \varepsilon) $,
		\begin{align}\label{ineq(1)}
			\qquad 
			\left| \frac{z^{s-1}e^{(v+w-\alpha)z}}{(e^{vz} -1)(e^{wz}-1)} \right|
			\leq
			M_1M_2 \cdot |z^{s-1}|e^{-T|z|} 
		\end{align}
		If $ 0 < y_1, y_2 < 1 $ or $ y_1 = y_2 = 0 $, then we can choose $T>0$. In fact, since $v,w$ are linear independent over $\Q$, for any $z'$ with $ |z'|=1 $ at least one 
		of $ \mathrm{Re}(vz') \neq 0 $ and $ \mathrm{Re}(wz') \neq 0 $ holds.
		If $ 0 < y_1 < 1 $ and $ y_2 = 0 $, that we can choose $T>0$. In fact, 
		\begin{align*}
			 \mathrm{Re}((vy_1+wy_2)z')- \mathrm{Re}(vz')_+ - \mathrm{Re}(wz')_+ 
		= \mathrm{Re}(vz') y_1 - \mathrm{Re}(vz')_+ - \mathrm{Re}(wz')_+ < 0
		\end{align*}
		Similarly, in the case of $ y_1 = 0 $ and $ 0 < y_2 < 1 $, we can choose $T>0$.
		
		From (\ref{ineq(1)}), we see that the integral term on the rightmost side of 
		(\ref{integral_CR}) converges to 0 when the radius of the contour goes to infinity if $ 0<y_1,y_2<1 $ and $ 0<y_1<1, y_2=0 $ and $ y_1=0, 0<y_2<1 $
		or $ y_1=y_2=0 $ with $ \mathrm{Re}(s)<0 $. Namely, 
		\begin{align*}
			\left|
			\int_{C_R(\theta)}
			\frac{z^{s-1}e^{(v+w-\alpha)z}}{(e^{vz} -1)(e^{wz}-1)} dz
			\right|
			&\leq
			\int_{|z|=R}
			\left| \frac{z^{s-1}e^{(v+w-\alpha)z}}{(e^{vz} -1)(e^{wz}-1)} \right|
			|dz|	\\
			&\leq
			\int_{|z|=R}
			M_1M_2 |z^{s-1}| e^{-T|z|} |dz|  
			\rightarrow 0
			\quad (R \rightarrow \infty)
		\end{align*}
		Hence we can calculate the integral by counting all the residues on the whole space. Since, by the assumption the poles of the integrand are all simple except at the origin, we obtain
		\begin{eqnarray}
			&& \zeta_2(s, \alpha(y_1,y_2); v,w)		\nonumber \\
			&& \ \ =- \frac{\Gamma(1-s)}{(2\pi i)^{1-s}e^{\pi is}}\left\{ \frac{1}{v^s} 
			\mathop{\sum_{n=-\infty}}\limits_{n \neq 0}^\infty \frac{e^{2\pi in(vy_1+wy_2)/v}}
			{(e^{2\pi inw/v}-1)n^{1-s}} + \frac{1}{w^s}
			\mathop{\sum_{n=-\infty}}\limits_{n \neq 0}^\infty \frac{e^{2\pi in(vy_1+wy_2)/w}}
			{(e^{2\pi inv/w}-1)n^{1-s}} \right\}. \nonumber \\
			\label{Thm1(i)}
		\end{eqnarray}
		Also, the two infinite series on the right-hand side converge absolutely and locally uniformly in $s \in \mathbb{C}$ if $(y_1,y_2)=(0,0)$, and in the region $\mathrm{Re}(s)<0$ if $ y_1 =0 $ or $ y_2=0 $, by Lemma \ref{lem_2nd}.
    			
	(ii) In the case when $ v, w $ are linearly dependent over $ \Q $, i.e., there exist $ p, q \in \N $ such that $ pv = qw $ and $ (p, q) = 1 $, $ F(z) $ has simple poles at	
		\[
		z = \frac{2\pi in}{v} \ \, (n \in \Z \setminus \{ 0 \},\ q \mid \hspace{-.65em}/  n) ,\ \
		\frac{2\pi in}{w} \ \, (n \in \Z \setminus \{ 0 \},\ p \mid \hspace{-.65em}/  n) .
		\]
		Moreover, the residues of $ F(z) $ at its simple poles are
		\begin{align*}
			\mathop{\mathrm{Res}}\limits_{z = 2\pi in/v} F(z) 
			= \frac{1}{v} \left(\frac{2\pi in}{v} \right)^{s-1}
			\frac{e^{(w - \alpha)2\pi in/v}}{e^{2\pi inw/v}-1}, \quad
			\mathop{\mathrm{Res}}\limits_{z = 2\pi in/w} F(z) 
			= \frac{1}{w} \left(\frac{2\pi in}{w} \right)^{s-1}
			\frac{e^{(v - \alpha)2\pi in/w}}{e^{2\pi inv/w}-1}
		\end{align*}
		On the other hand, for $ w = pv/q $,
		\begin{align*}
			& \lim_{z \rightarrow 2q\pi in/v} \frac{d}{dz} 
			\left\{ \left( z -\frac{2q\pi in}{v} \right)^2 
			\frac{z^{s-1}e^{(v+pv/q-\alpha)z}}{(e^{vz}-1)(e^{pvz/q}-1)} \right\}	\\
			& = \lim_{z_0 \rightarrow 0} \frac{d}{dz_0} 
			\left\{ z_0^2
			\left(z_0 + \frac{2q\pi in}{v}\right)^{s-1} 
			\frac{e^{(v+pv/q-\alpha)z_0}e^{-2q\pi in\alpha/v}}
			{(e^{vz_0}-1)(e^{pvz_0/q}-1)}	\right\}	\\
			& = \lim_{z_0 \rightarrow 0} \frac{d}{dz_0} 
			\Bigg\{z_0^2
			\left(z_0 + \frac{2q\pi in}{v}\right)^{s-1}	\\
			&  \ \  \qquad \qquad \qquad   \times \frac{e^{(v+pv/q-\alpha)z_0}e^{-2q\pi in\alpha/v}}
			{\left( vz_0 + \frac{v^2}{2!}z_0^2 + O(z_0^3) \right)
				\left( \frac{pv}{q}z_0 + \frac{1}{2!}\left(\frac{pv}{q}\right)^2 z_0^2 + O(z_0^3) \right)}
			\Bigg\}	\\
			& = \frac{q}{pv^2} (s-1) \left( \frac{2q\pi in}v \right)^{s-2} e^{-2q\pi in\alpha /v} +\left( \frac{p+q}{2pv} -\frac{\alpha q}{pv^2} \right)
			\left( \frac{2q\pi in}v \right)^{s-1} e^{-2q\pi in\alpha /v}.
		\end{align*}
		Therefore, $ F(z) $ has double poles at
		\[
		z = \frac{2\pi i qn}{v} = \frac{2\pi ipn}{w} \quad (n \in \Z \setminus \{ 0 \}).
		\]
		Next, we calculate the residue sum:
		\begin{align*}
			\sum_{0 < |n| \leq M} \mathrm{Res} F(z)
			= \mathop{\sum_{0 < |n| \leq L}}\limits_{q \; \mid \hspace{-.25em}/ \, n}
			\mathop{\mathrm{Res}}\limits_{z = 2\pi i n/v} F(z)
			+ \mathop{\sum_{0 < |n| \leq K}}\limits_{p \; \mid \hspace{-.25em}/ \, n}
			\mathop{\mathrm{Res}}\limits_{z = 2\pi i n/w} F(z)	
			+ \sum_{0 < |k| \leq K} \mathop{\mathrm{Res}}\limits_{z = 2\pi i qk/v} F(z)
		\end{align*}
		thus, we obtain
		\begin{align}
			& \zeta_2(s, \alpha; v,w)		\nonumber \\
			& \ \ =- \frac{\Gamma(1-s)}{(2\pi i)^{1-s}e^{\pi is}}
			\left\{ \frac{1}{v^s} 
			\mathop{\sum_{n=-\infty}}\limits_{q \; \mid \hspace{-.25em} / \, n}^\infty
			\frac{e^{2\pi in(w-\alpha)/v}}
			{(e^{2\pi inw/v}-1)n^{1-s}} + \frac{1}{w^s}
			\mathop{\sum_{n=-\infty}}\limits_{p \; \mid \hspace{-.25em}/ \, n}^\infty
			\frac{e^{2\pi in(v-\alpha)/w}}
			{(e^{2\pi inv/w}-1)n^{1-s}}  \right. \nonumber \\
			& \quad \qquad
			+ \frac{q^{s-1}}{2\pi ipv^s}(s-1) \mathop{\sum_{n=-\infty}}\limits_{n \neq 0}^\infty \frac{e^{-2\pi inq\alpha/v}}{n^{2-s}} 	
			\left. +\left( \frac{p+q}{2pv} -\frac{\alpha q}{pv^2} \right)
			\left( \frac{q}{v} \right)^{s-1} \mathop{\sum_{n=-\infty}}\limits_{n \neq 0}^\infty \frac{e^{-2\pi inq \alpha /v}}{n^{1-s}} 
			\right\}. \nonumber 
		\end{align}
		Moreover, the first and second infinite series of the right-hand side converge absolutely and uniformly on the whole space $ \C $ if $(y_1,y_2)=(0,0)$, and on the region $\mathrm{Re}(s)<0$ if $ y_1 =0 $ or $ y_2=0 $, by Lemma \ref{lem_3rd}.
		Furthermore, the third and fourth terms of the above are equal to
		\begin{align*}
			&\mathop{\sum_{n=-\infty}}\limits_{n \neq 0}^\infty 
			\frac{e^{-2\pi inq\alpha/v}}{n^{2-s}}
			=\zeta_L\left(2-s,1,-\frac{q\alpha}v\right)
			+ e^{\pi is}\zeta_L\left(2-s,1,\frac{q\alpha}v \right),	\\
			&\mathop{\sum_{n=-\infty}}\limits_{n \neq 0}^\infty
			\frac{e^{-2\pi inq \alpha /v}}{n^{1-s}}
			=\zeta_L\left(1-s,1,-\frac{q\alpha}v\right)
			- e^{\pi is}\zeta_L\left(1-s,1,\frac{q\alpha}v\right),
		\end{align*}
		respectively.
	Hence, the proof of Theorem \ref{th:Main_Theorem1} is complete.
	\end{proof}
	
	\medskip
	
\begin{proof}[Proof of Theorem \ref{th:convexity}]
Let
\[
B = \{ s \in \mathbb{C} \mid \sigma_1 \le \sigma \le \sigma_2,\ |t| \ge t_0 \}
\]
with $t_0>2$. We apply the Phragm\'en--Lindel\"of principle
to $g(s)=\zeta_2(s,\alpha;v,w)$ on the strip.
First, for $\sigma \ge 2+2\varepsilon$, the series defining
$\zeta_2(s,\alpha;v,w)$ converges absolutely, hence
\[
\zeta_2(s,\alpha;v,w) \ll 1.
\]

Next, we consider the region $\sigma<0$.

\medskip

\noindent
(i) In the case where $v/w$ is a real algebraic irrational number.
By the functional equation (Theorem~\ref{th:Main_Theorem1}(i))
and Lemma~\ref{lem_2nd}, the series appearing on the right-hand
side are absolutely convergent for $\sigma<0$. Hence,
using Stirling's formula
\[
\Gamma(1-s) \ll |t|^{1/2-\sigma} e^{-\pi |t|/2},
\]
we obtain
\[
\zeta_2(s,\alpha;v,w)
\ll \left|\frac{\Gamma(1-s)}{(2\pi i)^{1-s} e^{\pi i s}}\right|
\ll |t|^{1/2-\sigma}
\qquad (\sigma<0).
\]

\medskip

\noindent
(ii) In the case where $v/w \in \mathbb{Q}$.
By the functional equation (Theorem~\ref{th:Main_Theorem1}(ii)),
the dominant contribution comes from the term involving
$\zeta_L(2-s,1,-q\alpha/v)$.
Since the Lerch zeta-function remains bounded in vertical strips
with $\Re(2-s)>1$, the growth is determined by the factor
$\Gamma(1-s)$.
Hence, using Stirling's formula, we obtain
\[
\zeta_2(s,\alpha;v,w)
\ll |t|^{3/2-\sigma}
\qquad (\sigma<0).
\]
Now we apply the Phragm\'en--Lindel\"of principle to the strip
$-2\varepsilon \le \sigma \le 2+2\varepsilon$.
For case (i), we take
$
a = 1/2 + 2\varepsilon, \ b = 0,
$
and obtain
\[
\zeta_2(\sigma+it,\alpha;v,w)
\ll |t|^{(2-\sigma)/4 + \varepsilon}.
\]
For case (ii), we take
$
a = 3/2 + 2\varepsilon, \ b = 0,
$
and obtain
\[
\zeta_2(\sigma+it,\alpha;v,w)
\ll |t|^{3(2-\sigma)/4 + \varepsilon}.
\]
Hence, the proof of Theorem \ref{th:convexity} is complete.
\end{proof}

	\medskip
	
	\begin{proof}[Proof of Theorem \ref{Lindelof_zeta_2}]
	Assume that 
	\[
	\zeta_2(\sigma +it, \alpha; v,w) = O(|t|^{\varepsilon/2k})  \qquad 
	\left( \frac12 \leq \sigma \leq 2 \right),
	\]
	for any $ \varepsilon > 0 $ and $ k \in \N $. Then, 
	\[
	\int_1^T |\zeta_2(\sigma +it, \alpha; v,w)|^{2k} dt 
	\ll \int_1^T t^{\varepsilon} dt 
	= \frac1{1+\varepsilon} (T^{1+\varepsilon} - 1)
	\ll T^{1+\varepsilon}.
	\]
	Conversely, assume that
\[
\int_2^T |\zeta_2(\sigma+it,\alpha;v,w)|^{2k}dt
=O(T^{1+\varepsilon})
\]
holds for any $k\in\mathbb{N}$ and any $\varepsilon>0$.
Taking $\ell=2k$ in Lemma \ref{Lem:int}, we have
\begin{align*}
|\zeta_2(\sigma+it,\alpha;v,w)|^{2k}
&\ll (\log t)\int_{-\delta}^{\delta}
|\zeta_2(\sigma+i(t+x),\alpha;v,w)|^{2k}dx+t^{-A}.
\end{align*}
By the assumed mean-value estimate, the integral on the right-hand side is
\[
\int_{-\delta}^{\delta}
|\zeta_2(\sigma+i(t+x),\alpha;v,w)|^{2k}dx
\ll t^{1+\varepsilon}.
\]
Therefore,
\[
|\zeta_2(\sigma+it,\alpha;v,w)|^{2k}
\ll (\log t)t^{1+\varepsilon}+t^{-A}
\ll t^{1+2\varepsilon}.
\]
Hence
\[
|\zeta_2(\sigma+it,\alpha;v,w)|
\ll t^{1/2k+\varepsilon/k}.
\]
Since $k$ is arbitrary, this implies
\[
\zeta_2(\sigma+it,\alpha;v,w)=O(t^\varepsilon),
\]
for $1/2 \le \sigma \le 2$ and any $\varepsilon>0$.
	\end{proof}
    

\medskip

\begin{remark}
The proof of Theorem \ref{Lindelof_zeta_2} is not restricted to the case
$r=2$. Once an analogue of Lemma \ref{Lem:int} is established for the
Barnes multiple zeta-function
$\zeta_r(s,\alpha;w_1,\ldots,w_r)$,
the same argument yields the corresponding equivalence between the
Lindel\"of-type bound and the moment estimates for
$\zeta_r(s,\alpha;w_1,\ldots,w_r)$.
In the general case, additional phenomena related to the arithmetic
structure of the parameters $w_1,\ldots,w_r$ are expected to arise.
In particular, recent investigations suggest that the corresponding
Lindel\"of-type behavior may depend on the dimension of the
$\Q$-linear span of $w_1,\ldots,w_r$.
An extension of these results to the Barnes multiple zeta-function
will be considered in future work.
\end{remark}

\medskip

\section*{Acknowledgments}
First of all, I would like to express my deepest gratitude to my academic advisor
Prof. Kohji Matsumoto for his valuable advice and guidance. I also sincerely thank Prof. Takashi Nakamura and Prof. Hirotaka Akatsuka
for their valuable advice and for many useful conversations.
I would like to thank all members of the same laboratory and all members of the Kansai Multiple Zeta Study Group.

\bibliographystyle{plain}
\bibliography{References} 

@article{Arakawa1982,
 author = {Arakawa, Tsuneo},
 title = {Generalized eta-functions and certain ray class invariants of real quadratic fields},
 fjournal = {Mathematische Annalen},
 journal = {Math. Ann.},
 issn = {0025-5831},
 volume = {260},
 pages = {475--494},
 year = {1982},
 language = {English},
 doi = {10.1007/BF01457027},
 url = {https://eudml.org/doc/163683},
 zbMATH = {3749130},
 Zbl = {0477.12007}
}

@Article{Barnes1899,
 Author = {Barnes, E. W.},
 Title = {The genesis of the double gamma functions.},
 FJournal = {Proceedings of the London Mathematical Society},
 Journal = {Proc. Lond. Math. Soc.},
 ISSN = {0024-6115},
 Volume = {31},
 Pages = {358--381},
 Year = {1899},
 Language = {English},
 DOI = {10.1112/plms/s1-31.1.358},
 URL = {zenodo.org/record/1447742},
 zbMATH = {2668383},
 JFM = {30.0389.03}
}

@Article{Barnes1901,
 Author = {Barnes, E. W.},
 Title = {The theory of the double gamma function.},
 FJournal = {Philosophical Transactions of the Royal Society of London, Series A},
 Journal = {Philos. Trans. R. Soc. Lond., Ser. A, Contain. Pap. Math. Phys. Character},
 Volume = {196},
 Pages = {265--387},
 Year = {1901},
 Language = {English},
 DOI = {10.1098/rsta.1901.0006},
 zbMATH = {2662859},
 JFM = {32.0442.02}
}

@Misc{Barnes1904,
 Author = {Barnes, E. W.},
 Title = {On the theory of the multiple {Gamma} function.},
 Year = {1904},
 Language = {English},
 HowPublished = {Cambr. {Trans}. 19, 374-425 (1904).},
 zbMATH = {2653713},
 JFM = {35.0462.01}
}

@article{BombieriIwaniec1986,
 author = {Bombieri, E. and Iwaniec, H.},
 title = {On the order of {{\(\zeta(+it)\)}}},
 fjournal = {Annali della Scuola Normale Superiore di Pisa. Classe di Scienze. Serie IV},
 journal = {Ann. Sc. Norm. Super. Pisa, Cl. Sci., IV. Ser.},
 issn = {0391-173X},
 volume = {13},
 pages = {449--472},
 year = {1986},
 language = {English},
 url = {https://eudml.org/doc/83987},
 zbMATH = {3995796},
 Zbl = {0615.10047}
}

@article{Bourgain2017,
 author = {Bourgain, J.},
 title = {Decoupling, exponential sums and the {Riemann} zeta function},
 fjournal = {Journal of the American Mathematical Society},
 journal = {J. Am. Math. Soc.},
 issn = {0894-0347},
 volume = {30},
 number = {1},
 pages = {205--224},
 year = {2017},
 language = {English},
 doi = {10.1090/jams/860},
 zbMATH = {6640529},
 Zbl = {1352.11065}
}

@article{Huxley2005,
 author = {Huxley, M. N.},
 title = {Exponential sums and the {Riemann} zeta function. {V}},
 fjournal = {Proceedings of the London Mathematical Society. Third Series},
 journal = {Proc. Lond. Math. Soc. (3)},
 issn = {0024-6115},
 volume = {90},
 number = {1},
 pages = {1--41},
 year = {2005},
 language = {English},
 doi = {10.1112/S0024611504014959},
 zbMATH = {2145417},
 Zbl = {1083.11052}
}

@misc{Ivic1985,
 author = {Ivi{\'c}, Aleksandar},
 title = {The {Riemann} zeta-function. {The} theory of the {Riemann} zeta-function with applications},
 year = {1985},
 language = {English},
 howpublished = {A {Wiley}-{Interscience} {Publication}. {New} {York} etc.: {John} {Wiley} \& {Sons}. {XVI}, 517 p. {{\textsterling}} 57.80 (1985).},
 zbMATH = {3886001},
 Zbl = {0556.10026}
}

@Article{KomoriMatsumotoTsumura2013,
 Author = {Komori, Yasushi and Matsumoto, Kohji and Tsumura, Hirofumi},
 Title = {Barnes multiple zeta-functions, {Ramanujan}'s formula, and relevant series involving hyperbolic functions},
 FJournal = {Journal of the Ramanujan Mathematical Society},
 Journal = {J. Ramanujan Math. Soc.},
 ISSN = {0970-1249},
 Volume = {28},
 Number = {1},
 Pages = {49--69},
 Year = {2013},
 Language = {English},
 zbMATH = {6158510},
 Zbl = {1317.11092}
}

@Book{LaurincikasGarunkstis2002,
 Author = {Laurin{\v{c}}ikas, Antanas and Garunk{\v{s}}tis, Ram{\~u}nas},
 Title = {The {Lerch} zeta-function},
 ISBN = {1-4020-1014-1},
 Year = {2002},
 Publisher = {Dordrecht: Kluwer Academic Publishers},
 Language = {English},
 zbMATH = {1924169},
 Zbl = {1028.11052}
}

@Article{Miyagawa2018,
 Author = {Miyagawa, Takashi},
 Title = {Mean values of the {Barnes} double zeta-function},
 FJournal = {Tokyo Journal of Mathematics},
 Journal = {Tokyo J. Math.},
 ISSN = {0387-3870},
 Volume = {41},
 Number = {2},
 Pages = {557--572},
 Year = {2018},
 Language = {English},
 DOI = {10.3836/tjm/1502179261},
 zbMATH = {7053492},
 Zbl = {1448.11169}
}

@article{Miyagawa2019,
 author = {Miyagawa, Takashi},
 title = {Approximate functional equation and upper bounds for the {Barnes} double zeta-function},
 fjournal = {JP Journal of Algebra, Number Theory and Applications},
 journal = {JP J. Algebra Number Theory Appl.},
 issn = {0972-5555},
 volume = {43},
 number = {2},
 pages = {155--183},
 year = {2019},
 language = {English},
 doi = {10.17654/NT043020155},
 zbMATH = {7262260},
 Zbl = {1444.11185}
}

@article{Ruijsenaars2000,
 author = {Ruijsenaars, S. N. M.},
 title = {On {Barnes}' multiple zeta and gamma functions},
 fjournal = {Advances in Mathematics},
 journal = {Adv. Math.},
 issn = {0001-8708},
 volume = {156},
 number = {1},
 pages = {107--132},
 year = {2000},
 language = {English},
 doi = {10.1006/aima.2000.1946},
 zbMATH = {1579914},
 Zbl = {0966.33013}
}

@article{Spreafico2009,
 author = {Spreafico, M.},
 title = {On the {Barnes} double zeta and gamma functions},
 fjournal = {Journal of Number Theory},
 journal = {J. Number Theory},
 issn = {0022-314X},
 volume = {129},
 number = {9},
 pages = {2035--2063},
 year = {2009},
 language = {English},
 doi = {10.1016/j.jnt.2009.03.005},
 zbMATH = {5575276},
 Zbl = {1185.11055}
}

@Book{SrivastavaChoi2001,
 Author = {Srivastava, H. M. and Choi, Junesang},
 Title = {Series associated with the zeta and related functions},
 ISBN = {0-7923-7054-6},
 Year = {2001},
 Publisher = {Dordrecht: Kluwer Academic Publishers},
 Language = {English},
 zbMATH = {1651797},
 Zbl = {1014.33001}
}

@Misc{Titchmarsh1986,
 Author = {Titchmarsh, E. C.},
 Title = {The theory of the {Riemann} zeta-function. 2nd ed., rev. by {D}. {R}. {Heath}-{Brown}},
 Year = {1986},
 Language = {English},
 HowPublished = {Oxford {Science} {Publications}. {Oxford}: {Clarendon} {Press}. x, 412 pp. },
 zbMATH = {3968684},
 Zbl = {0601.10026}
}

\end{document}